\documentclass[reqno]{amsart}
\usepackage[usenames,dvipsnames]{color}
\usepackage[T1]{fontenc}
\usepackage[utf8]{inputenc}
\usepackage{lmodern}
\usepackage{amsmath,amssymb,mathtools}
\usepackage{microtype}
\usepackage{enumitem}
\usepackage{needspace}
\usepackage{amsthm}
\usepackage{amsmath}
\usepackage{amssymb}
\usepackage{amscd}
\usepackage{graphics}
\usepackage{latexsym}
\usepackage{stmaryrd}
\usepackage{empheq}
\usepackage{xcolor}
\usepackage{tikz}
\usetikzlibrary{arrows.meta}
\usepackage[colorlinks=true,linkcolor=blue!55!black,
 citecolor=orange,urlcolor=red!55!black]{hyperref}
\usepackage{enumitem}
\usepackage[style=trad-abbrv,maxnames=99,maxalphanames=9, isbn=false, giveninits=true, doi=false, url=true]{biblatex}
\bibliography{ref.bib}
\renewbibmacro{in:}{}
\theoremstyle{remark}

\definecolor{titlecol}{named}{BrickRed}
\definecolor{headcol}{named}{Violet}
\definecolor{seccol}{named}{Red}
\definecolor{sseccol}{named}{Bittersweet}
\definecolor{pbcol}{named}{Black}
\definecolor{sncol}{named}{Brown}
\definecolor{acol1}{named}{Red}
\definecolor{acol2}{named}{Apricot}

\def\tr{\text{tr}}

\newcommand{\cab}{c_{\alpha,\beta}}

\def\W{{\mathfrak W}}

\def\W{{\mathfrak W}}

\theoremstyle{plain}
\newtheorem{thm}{Theorem}

\newtheorem{lemma}[thm]{Lemma}

\newtheorem{theorem}[thm]{Theorem}

\newtheorem{definition}[thm]{Definition}
\newtheorem{remark}[thm]{Remark}

\allowdisplaybreaks[1]
\makeatother
\title[Finiteness for the J-equation]{Finiteness of null subvarieties and optimal destabilizers for the \(J\)-equation}
\date{}
\author{Junbang Liu}
\address{Department of Mathematics, The Hong Kong University of Science and Technology, Clear Water Bay, Kowloon, Hong Kong}
\email{junbangliu@ust.hk}
\hypersetup{pdftitle={Finiteness of null subvarieties and optimal destabilizers for the J-equation}}
\begin{document}
\begin{abstract}
    Let $X$ be a compact K\"ahler manifold and $(\alpha,\beta)$ be a pair of K\"ahler classes. We show that when $(\alpha,\beta)$ is $J$-semistable, the collection of positive-dimensional irreducible $J$-null subvarieties is finite. We also obtain a uniform positive lower bound for the normalized $J$-slope of every irreducible non-$J$-null subvariety. This confirms the finiteness and uniform-gap expectation of Khalid–Sjöström Dyrefelt \cite[Section 1.2, Theorem 1.5; Remark 1.8]{KSD26}.  As a corollary, we show that the collection of optimal destabilizers is finite when the stability threshold is nonpositive. 
\end{abstract}
\maketitle
\section{Introduction}

Let $X$ be a compact K\"ahler manifold of dimension $n$ with a pair of K\"ahler classes $(\alpha,\beta)$. Fix a K\"ahler form $\chi\in \beta$, the $J$-equation seeks a K\"ahler form $\omega\in \alpha$ such that \[
n\omega^{n-1}\wedge \chi=\cab\omega^n, \qquad \cab=n\beta\alpha^{n-1}/\alpha^n.
\]
It was introduced in the moment map picture of Donaldson \cite{D99}, and in the study of the Mabuchi energy by Chen \cite{C00}. The solution is the critical point of the corresponding $J$-functional and it is related to the properness of the Mabuchi energy, which leads to a criteria for existence of cscK metrics. Since then, there have been extensive studies of the $J$-equation and the $J$-flow, including convergence, singularity formation, and more general inverse Hessian equations; see \cite{C04,W04,W06,SW08,FLM11,FL12,FL13,SW13,FLSW14,LS15,CS17,S18,SD20,S20,C21,DP21,T23,FM24,GS24,KSD24,KSD26,M26a,M26b} and the references therein. The analytic theory relates the solvability to a strict pointwise cone condition, established by Weinkove \cite{W04,W06} and Song-Weinkove \cite{SW08}. The relations between the $J$-equation and numerical inequalities on intersection numbers was conjectured by Lejmi-Sz\'ekelyhidi \cite{LS15}, proved in the toric case by Collins-Sz\'ekelyhidi \cite{CS17}. G. Chen's breakthrough work \cite{C21} proved the equivalence of uniform numerical intersection inequalities with the pointwise cone condition. Datar-Pingali \cite{DP21} and Song \cite{S20} obtained the full equivalence, in the projective manifold and K\"ahler space respectively. In the normalization above, the numerical intersection inequalities are \[
\int_V\cab\alpha^{p}-p\beta\alpha^{p-1}>0, \qquad \text{ for every subvariety }V\subset X, 0<\dim V<n. 
\]
Replacing these strict inequalities by nonnegative ones leads to numerical \(J\)-nefness condition. It's also called \(J\)-semistability. The subvarieties on which equality holds obstruct smooth solvability and are natural candidates for describing degeneration of the equation. On K\"ahler surfaces, Fang--Lai--Song--Weinkove \cite{FLSW} proved smooth convergence away from finitely many curves of negative self-intersection under a semipositive boundary representative assumption. Murakami \cite{M26a,M26b}  obtained convergence in the sense of currents under the numerical \(J\)-nefness condition. A systematic study of finite numerical tests was undertaken by Khalid--Sj\"ostr\"om Dyrefelt \cite[Theorems 1.10 and 5.3; Proposition 5.8]{KSD24}, who established wall--chamber decompositions on compact K\"ahler surfaces and characterized optimal destabilizing curves for the \(J\)-equation. Their higher-dimensional work \cite[Theorems 1.1, 1.5 and 1.11]{KSD26}, building on Sj\"ostr\"om Dyrefelt's study of stability thresholds \cite[equation (3) and Theorem 5]{SD20}, proves finiteness and wall--chamber results for generalised Monge--Amp\`ere equations under positivity assumptions on the relevant factor classes. The finiteness and uniform-gap questions addressed here originate in this programme, especially \cite[Section 1.2, ; Remark 1.8]{KSD26}, where the uniform gap is predicted and established conditionally; see \cite[Remark 1.8 and Theorem 3.10]{KSD26}. In the \(J\)-nef setting, their questions ask whether there are only finitely many null subvarieties and whether the normalized intersections of all remaining subvarieties are bounded uniformly away from zero.

We now fix the notations. Let  \(X\) be a connected compact K\"ahler manifold,  and $(\alpha,\beta)$ be a pair of K\"ahler classes. Every subvariety is understood to be reduced and irreducible unless stated otherwise. The constant \(c\) in this paper is always the constant $\cab$. We use the following numerical form of the stability threshold introduced and studied by Sjöström
Dyrefelt \cite{SD20} and subsequently used in \cite{KSD24,KSD26,SSD26}.

\begin{definition}\label{def:numerical}For a proper positive-dimensional reduced irreducible analytic subvariety \(V\), put \(p=\dim V\) and
\[
c=\cab=n\beta\alpha^{n-1}/{\alpha^n}>0,
\]
\[
\mu(V,\alpha,\beta)=p\frac{\int_V\beta\alpha^{p-1}}{\int_V\alpha^p},\qquad
J(V,\alpha,\beta)=\int_V(c\alpha^p-p\beta\alpha^{p-1}).
\]
\begin{equation}\label{eq:threshold}
\begin{aligned}
r(V,\alpha,\beta)&=\frac{\cab-\mu(V,\alpha,\beta)}{n-p}
=\frac{J(V,\alpha,\beta)}{(n-p)\int_V\alpha^p},\\
\Gamma&=\inf_{V\subsetneq X,\ 0<\dim V<n}r(V,\alpha,\beta).
\end{aligned}
\end{equation}
The infimum in defining $\Gamma$ is defined to be $+\infty$ when the collection of proper positive-dimensional subvarieties is empty.  \(J\)-nefness means \(J(V,\alpha,\beta)\ge0\) for every such \(V\); a null subvariety has \(J(V,\alpha,\beta)=0\). When $\Gamma\leq 0$, $V$ is called an optimal destabilizer if it satisfies \(r(V,\alpha,\beta)=\Gamma\). 

\end{definition}

Our main result treats the finiteness of the collection of null subvarieties and the uniform gap away from them. Khalid–Sjöström Dyrefelt proved these conclusions under additional positivity assumptions on the relevant factor classes \cite[Theorem 1.5, Remark 1.8 and Theorem 3.10]{KSD26}. Theorem 2 removes these additional assumptions, confirming their expectation.

\begin{samepage}
\begin{theorem}\label{thm:null-finiteness}Assume the pair $(\alpha,\beta)$ is $J$-nef. Then the collection of null subvarieties is finite. Moreover, there is a constant \(\delta>0\), depending on \((X,\alpha,\beta)\), such that every non-null positive-dimensional proper subvariety of dimension $p$ satisfies
\[
J(V,\alpha,\beta)\ge\delta\int_V\beta^p.
\]
Thus the irreducible subvarieties with normalized slope below \(\delta\) are precisely the finitely many null subvarieties.
\end{theorem}
\end{samepage}

In particular, the finiteness of the null subvarieties implies that the union of them forms an analytic subset of $X$. The analyticity was proved by Fu \cite{F26} for toric manifolds. 
Our proof closely follows and extends the analytic method in the very recent work of Sivaram--Sj\"ostr\"om Dyrefelt \cite[Section 4]{SSD26}.  When \(\Gamma\le0\), they prove that all optimal destabilizers are contained in a fixed proper analytic subset and that the collections of optimal curves and optimal divisors are finite. At the semistable boundary \(\Gamma=0\), optimal subvarieties are precisely the null subvarieties. Their results therefore give finiteness of null curves and divisors in every dimension. 
In particular, on threefolds, their results give finiteness of all proper positive-dimensional null subvarieties. The threefold case was also proved in \cite{L26a} from a different perspective, where it identifies the numerical null locus of a $J$-nef pair with the analytic non-J-ample locus. Higher-dimensional results were established in \cite{L26b} with technical assumptions of $J$-bigness. Fu--Zhang \cite{FZ} proved a uniform modified-nef estimate, which verifies the numerical characterization of \(J\)-bigness condition in \cite[Proposition 2]{L26b}, and therefore removed that technical assumption. Hence, under the condition of $J$-nefness alone, the remaining question is finiteness for subvarieties of intermediate dimension \(2\le p\le n-2\), when $\dim X\geq 4$. Our theorem confirms this finiteness. 

Beyond the semistable case, one can ask which subvarieties give the strongest numerical obstruction and whether the optimal value is attained. The threshold \(\Gamma\) in \eqref{eq:threshold} arises in Sjöström Dyrefelt's study of the optimal coercivity constant of Donaldson's \(J\)-functional \cite{SD20} and in Khalid–Sjöström Dyrefelt’s study of optimal destabilizing curves in \cite{KD24}. Datar--Mete--Song \cite{DMS} proposed a complementary theory of canonical singular solutions and bubbling governed by minimal slopes, and established it on surfaces and in examples with symmetry. Fu \cite{FuMinimal} proved their characterization of semistability by the minimal birational \(J\)-slope. That invariant uses classes on modifications, whereas \(\Gamma\) is defined by subvarieties of \(X\). The next theorem deals with the finiteness question in this setting.

\begin{theorem}\label{thm:optimal-finiteness}
For a K\"ahler pair $(\alpha,\beta)$ with \(\Gamma\le0\), the proper positive-dimensional optimal destabilizers are nonempty and finite in all dimensions. There is also \(\sigma>0\) such that every nonoptimal positive-dimensional proper subvariety satisfies \(r(V,\alpha,\beta)\ge\Gamma+\sigma\).
\end{theorem}

The restriction \(\Gamma\le0\) is essential. For the hyperplane class on \(\mathbb P^n\), the pair \(\alpha=\beta\) has \(\Gamma=1\), and every proper positive-dimensional subvariety is a minimizer.

\subsection{Outline of the proof}
The proof follows the induction strategy of Khalid--Sj\"ostr\"om Dyrefelt \cite[proofs of Theorems 3.1 and 3.9]{KSD26} and the analytic construction of Sivaram--Sj\"ostr\"om Dyrefelt \cite{SSD26}. 
The key step is Lemma~\ref{lem:local-gap}. For each irreducible \(Z\subseteq X\) of dimension at least two, it provides a proper analytic subset \(E_Z\subsetneq Z\) and positive constants \(\delta_{Z,p}\) such that
\[
J(V,\alpha,\beta)\ge\delta_{Z,p}\int_V\beta^p
\]
whenever \(V\subset Z\) has dimension \(p<\dim Z\) and is not contained in \(E_Z\). 

To prove this estimate, we resolve \(Z\) and construct strict approximating forms by perturbing the two pullback classes at independently chosen rates. Chen's auxiliary equation \cite[Theorem 1.14]{C21} and the diagonal mass concentration method of Demailly--P\u{a}un \cite{DP04} then produce a positive current in a class smaller than the first pullback class. Eigenvalue truncation allows the diagonal contribution to be subtracted while retaining the required cone inequality. A further perturbation gives a current with a fixed positive lower bound and a uniform strict cone margin.

The local potentials of this current can be regularized and glued outside a fixed positive Lelong level set. The regularization radius may depend on the approximating form and on the compact set where a uniform estimate is required; the lower bound and cone margin do not. This is mainly inspired by the method in \cite{SSD26}, which is carried out on a fixed smooth K\"ahler manifold. Here we need to handle singular K\"ahler space.  Integration over a fixed strict transform, followed by passage to the cohomological limit and compact exhaustion, gives Lemma~\ref{lem:local-gap}. Iterating over the finitely many irreducible components of each \(E_Z\) proves Theorem~\ref{thm:null-finiteness}. Finally, the class-shift reduction of \cite[Lemma 14]{SD20}, \cite[proof of Theorem 3.10 and Remark 3.11]{KSD26}, and \cite[Section 2.2.1]{SSD26}, given by \(\beta_0=\beta-\Gamma\alpha\), reduces Theorem~\ref{thm:optimal-finiteness} to the \(J\)-nef case.

\textbf{Acknowledgements}. The research is partially supported by the Hong Kong General Research Fund \#16305625 of the Hong Kong Research Grants Council.

\textbf{Declaration on the use of AI}. All the ideas in the proof are the author's own, and ChatGPT is used to assist in the preparation of the
manuscript.  The author takes full responsibility for the paper’s content and correctness.

\section{Preliminaries}

Products of differential forms will be written without wedge signs when no ambiguity is possible. We use the same notation for the intersection products of their classes. For a positive current, say $T$, its absolutely continuous part is denoted by the subscript \(\mathrm{ac}\), i.e.  $T_{\mathrm ac}$; characteristic functions acting on its coefficient measures are denoted by \(\mathbf 1_E\).

\subsection{The cone function}
\begin{lemma}\label{lem:matrix-cone}
On an \(m\)-dimensional compact K\"ahler manifold, \(m\ge2\), for two closed real $(1,1)$-forms, \(h>0\) and \(b\ge0\), write
\[
P_b(h)=\max_{\dim_{\mathbb C}L=m-1}
\operatorname{tr}_{h|_L}(b|_L).
\]
The cone function $P_b(h)$ is convex and order-reversing in $h$ and monotone in $b$. If \(P_b(h)\le c-\eta\), then \(ch^p-pbh^{p-1}\ge\eta h^p\) for every \(1\le p<m\). 
\end{lemma}
Note that here we don't require $b$ to be K\"ahler but only semipositive.

\subsection{The strict numerical criterion}

\begin{theorem}[Theorem 1.1 in \cite{S20}]\label{thm:song}
On a compact K\"ahler \(d\)-fold with K\"ahler classes \(C,B\), if
   \[
C^d-dC^{d-1}B\ge0,\qquad
   \int_V(C^q-qC^{q-1}B)>0\quad(1\le q<d),
\]
   then for every K\"ahler form \(b\in B\) there is a K\"ahler form \(h\in C\) with
   \[
h^{d-1}-(d-1)h^{d-2}b>0.
\]
\end{theorem}
This is the smooth case of Song's results \cite[Theorem 1.1]{S20}. Song's theorem also applies to a compact analytic subvariety with ambient K\"ahler classes. Only the smooth case is used here.

\subsection{Auxiliary equations}\label{sec:analytic-inputs}

\begin{theorem}[Theorem 1.14 in \cite{Chen21}]\label{thm-Chen}
    On a compact K\"ahler \(N\)-fold, for K\"ahler forms \(\chi,\omega_0\), a constant \(D>0\), and a smooth function
\[
F>-\frac1{2ND^{N-1}},\qquad
\int F\chi^N=D[\omega_0]^N-N[\chi][\omega_0]^{N-1}\ge0,
\]
if \(P_\chi(\omega_0)<D\), then there is a K\"ahler form \(\omega_\varphi\in[\omega_0]\) satisfying
\[
\operatorname{tr}_{\omega_\varphi}\chi+
F\chi^N/\omega_\varphi^N=D,\qquad P_\chi(\omega_\varphi)<D.
\]
\end{theorem} 
This is \cite[Theorem 1.14]{C21}. Chen proves the auxiliary equation by two continuity paths: first deform \(\chi\) from a multiple of the initial metric with a nonnegative constant right-hand term, then deform that term to \(F\). The initial cone supplies the subsolution, and the explicit negative lower bound on \(F\) preserves ellipticity and the estimates \cite[Section 2]{C21}. We apply this theorem only to smooth K\"ahler data satisfying the stated integral identity. The estimates needed for the degenerate limit will be obtained below.

\subsection{Mass concentration}\label{sec:mass-concentration} For a compact K\"ahler \(N\)-fold \((M,\Theta)\) and a codimension-\(p\) analytic subvariety \(D\), there are a family of K\"ahler forms, parametrized by $\eta$,  \(\Theta_\eta\in[\Theta]\), \(\Theta_\eta\ge\Theta/2\), and tubular neighborhoods \(U_\eta\downarrow D\) as $\eta\to 0$, such that, for each suitable open \(U\) meeting the regular part of \(D\),
\[
\int_{U\cap U_\eta}\Theta_\eta^p\Theta^{N-p}\ge b_U>0.
\]
They arise from \(\psi_\eta=\frac12\log(\sum\rho_i^2|g_i|^2+\eta^2)\), with uniformly bounded negative Hessian, by adding a fixed small multiple of \(dd^c\psi_\eta\) to \(\Theta\). See \cite[Lemma 2.1(iii)--(iv)]{DP}; the associated mass concentration argument is developed in \cite[Proposition 2.6]{DP}.

\section{Approximation and concentration on a resolution}

\begin{samepage}
\begin{lemma}\label{lem:ambient-approximation}
For a \(J\)-nef pair $(\alpha,\beta)$ and every \(t>1\), there is a K\"ahler form \(\Omega_t\in t\alpha\) such that, for the fixed K\"ahler form \(b\in\beta\),
\[
c\Omega_t^q-q\Omega_t^{q-1}b>0\qquad(1\le q\le n-1).
\]

\end{lemma}
\end{samepage}

\begin{proof}
Apply Song's theorem~\ref{thm:song} to \(C=ct\alpha\) and \(B=\beta\). For each proper \(q\)-fold \(V\),
\[
\int_V((ct\alpha)^q-q(ct\alpha)^{q-1}\beta)
=c^{q-1}t^{q-1}\left(J(V,\alpha,\beta)+c(t-1)\int_V\alpha^q\right)>0.
\]
The top intersection equals \(c^nt^{n-1}(t-1)\alpha^n>0\). Thus a form \(h_t\in ct\alpha\) satisfies the strict \((n-1)\)-cone inequality. Set \(\Omega_t=h_t/c\). In a unitary frame, the cone inequality says that the sum of the eigenvalues of \(b\) relative to \(\Omega_t\) over any \((n-1)\)-plane is smaller than \(c\). The same holds for every smaller plane: extend it to an \((n-1)\)-plane and use positivity of the omitted trace. This proves all lower-degree strong inequalities.
\end{proof}

\begin{samepage}
Assume  \(J\)-nefness of $(\alpha,\beta)$. We fix two K\"ahler form $a_X\in \alpha, b_X\in \beta$. Let \(g:Y\to Z\subseteq X\) be a resolution of an irreducible analytic subvariety of dimension \(m\ge2\). Here we allow \(Z=X\) and \(g=\mathrm{id}\). Put \(a=g^*a_X\), \(b=g^*b_X\), \(A=[a]\), \(B=[b]\), and fix a K\"ahler form \(\theta\) on \(Y\). Then \(A^m,B^m>0\), the forms \(a,b\) are semipositive and positive off a proper analytic set.
\begin{lemma}\label{lem:resolution-approximation}
 There are
\[
A_j=t_jA+u_j[\theta],\qquad B_j=B+v_j[\theta],\qquad
t_j\downarrow1,\quad u_j,v_j\downarrow0,
\]
with K\"ahler representatives \(h_j\in A_j\), \(b_j=b+v_j\theta\) satisfying
\begin{equation}\label{eq:resolved-cone}
P_{b_j}(h_j)<c,\qquad cA_j^m-mB_jA_j^{m-1}>0.
\end{equation}
\end{lemma}
\end{samepage}

\begin{proof}
The resolution map is an immersion over \(Z_{\rm reg}\), so the pullback forms are positive there. The top gap is
\[
cA^m-mBA^{m-1}=J(Z,\alpha,\beta )\ge0
\]
when \(Z\ne X\), and is zero when \(Z=X\) by the definition of \(c\).

Pull back \(\Omega_t\) from Lemma~\ref{lem:ambient-approximation} and write \(h_t=g^*\Omega_t\). It is semipositive, and \(ch_t^q-qbh_t^{q-1}\ge0\) for \(1\le q<m\). Choose \(M_t\) with \(h_t\le M_t\theta\). For \(u>0\), choose
\[
0<v\le\frac c2\min_{1\le q\le m}\frac{u^q}{q(M_t+u)^{q-1}}.
\]
The exact identity
\begin{equation}\label{eq:perturbation}
\begin{aligned}
&c(h_t+u\theta)^q-q(b+v\theta)(h_t+u\theta)^{q-1}\\
&\quad=\sum_{\ell=0}^{q-1}\binom q\ell u^\ell\theta^\ell
\bigl(ch_t^{q-\ell}-(q-\ell)bh_t^{q-\ell-1}\bigr)\\
&\qquad+cu^q\theta^q-qv(h_t+u\theta)^{q-1}\theta.
\end{aligned}
\end{equation}
has last line at least \(cu^q\theta^q/2\). For \(q<m\) its other summands are semipositive, proving the strict pointwise cone inequality. For \(q=m\) the summands with \(\ell\ge1\) are semipositive, while the integral of the \(\ell=0\) term is
\[
t^{m-1}\bigl(J(Z,\alpha,\beta)+c(t-1)A^m\bigr)>0.
\] Thus the top integral is positive too. Choose any \(t_j\downarrow1,u_j\downarrow0\), and then \(v_j\) satisfying the bound and \(v_j<v_{j-1}/2\). Set \(h_j=h_{t_j}+u_j\theta\). This proves \eqref{eq:resolved-cone}. 
\end{proof}

The next lemma is the mass concentration trick used by Chen \cite{C21}.  We present the details here to handle the approximations on the resolutions and the cone inequality for semipositive form. 

\begin{samepage}
\begin{lemma}\label{lem:concentration}
There is a closed positive current \(S\in A-\varepsilon B\), with fixed \(\varepsilon>0\), satisfying the cone inequality \(P_b(S)\le c\) in Chen's sense of local convolution,  wherever \(b>0\). 
\end{lemma}
\end{samepage}

\begin{proof}
\emph{Step 1. The auxiliary equation on the product.}

By Lemma~\ref{lem:resolution-approximation}, we can apply Chen's theorem to obtain \(\omega_j\in A_j\) with
\[
\operatorname{tr}_{\omega_j}b_j+
f_j b_j^m/\omega_j^m=c,\qquad
f_j=(cA_j^m-mA_j^{m-1}B_j)/B_j^m>0,
\]
where \(b_j=b+v_j\theta\). In particular the full trace $\tr_{\omega_j}b_j<c$.

On \(M=Y\times Y\), put
\[
N=2m,\quad D=c+m,\quad
\chi_j=\pi_1^*b_j+\pi_2^*b_j,\quad
L_j=\pi_1^*A_j+\pi_2^*B_j.
\]
Here $\pi_1$ is the projection of $Y\times Y$ to the first $Y$, and similarly for $\pi_2$.
The product \(\omega_j\oplus b_j\) has full trace \(<D\). Direct product expansion gives the exact cancellation
\[
f_j^0:=\frac{DL_j^N-N[\chi_j]L_j^{N-1}}{[\chi_j]^N}
=\frac{cA_j^m-mA_j^{m-1}B_j}{B_j^m}=f_j>0.
\]
The denominator tends to \(\binom{2m}{m}(B^m)^2>0\).

Take the logarithmic regularizations (discussed in section~\ref{sec:mass-concentration}) \(\Theta_\eta\) for the diagonal subvariety $\Delta:=\{(y,y)\in Y\times Y\},$ where $\Theta=\pi_1^*\theta+ \pi_2^*\theta$.  Choose a fixed small \(s>0\) and let
\[
\chi_{j,s,\eta}=(1-s)\chi_j+s\Theta_\eta,\qquad
r_j(s)=\frac{((1-s)[\chi_j]+s[\Theta])^N}{[\chi_j]^N}.
\]
The classes vary in a compact set and the denominator is bounded below, so \(r_j(s)=1+O(s)\) uniformly. Choose \(s\) so \(r_j(s)\le2\). This $s$ remains fixed as \(\eta\to0\).

Set
\[
q=\frac1{16ND^{N-1}},\qquad
F_{j,\eta}=f_j^0+q\left(\frac{\chi_{j,s,\eta}^N}{\chi_j^N}-r_j(s)\right).
\]
Then \(F_{j,\eta}>-2q>-1/(2ND^{N-1})\), and its integral is exactly the required positive compatibility value. Chen's theorem~\ref{thm-Chen} gives \(\Xi_{j,\eta}\in L_j\) with
\[
P_{\chi_j}(\Xi_{j,\eta})<D,\qquad
D\Xi_{j,\eta}^N=N\chi_j\Xi_{j,\eta}^{N-1}
+F_{j,\eta}\chi_j^N.
\]
Every eigenvalue of \(\chi_j\) relative to \(\Xi_{j,\eta}\) is \(<D\), so \(\Xi_{j,\eta}>D^{-1}\chi_j\). Therefore
\begin{equation}\label{eq:product-volume}
D\Xi_{j,\eta}^N
\ge q\chi_{j,s,\eta}^N+
\left(ND^{1-N}+f_j^0-qr_j(s)\right)\chi_j^N
\ge q\chi_{j,s,\eta}^N.
\end{equation}

\emph{Step 2. Uniform diagonal mass.}

Write \(\lambda_1\le\cdots\le\lambda_{2m}\) for the eigenvalues of \(\Xi_{j,\eta}\) relative to \(\chi_{j,s,\eta}\). Equation \eqref{eq:product-volume} gives \(\prod_i\lambda_i\ge q/D\). Cohomology gives
\[
\int \lambda_{m+1}\cdots\lambda_{2m}\chi_{j,s,\eta}^{2m}
\le(2m)!\int\Xi_{j,\eta}^m\chi_{j,s,\eta}^m\le M_0
\]
uniformly in \(j,\eta\). Also \(\chi_{j,s,\eta}\ge(s/2)\Theta\), and
\[
\int_{U\cap U_\eta}\chi_{j,s,\eta}^m\Theta^m\ge s^mb_U.
\]
The set where \(\lambda_{m+1}\cdots\lambda_{2m}>M_0/\delta\) has
\(\chi_{j,s,\eta}^{2m}\)-volume at most \(\delta\), hence
\(\chi_{j,s,\eta}^m\Theta^m\)-mass at most \((2/s)^m\delta\).
Outside it,
\(\Xi_{j,\eta}^m\ge(q\delta/(DM_0))\chi_{j,s,\eta}^m\).
Take \(\delta>0\) so \((2/s)^m\delta<s^mb_U/2\). This gives a fixed positive lower bound on
\(\int_{U\cap U_\eta}\Xi_{j,\eta}^m\Theta^m\).

For each \(j\), take \(\eta\downarrow0\) along a subsequence with
\(\Xi_{j,\eta}^m\rightharpoonup Q_j\). Skoda's restriction theorem and the support theorem give
\begin{equation}\label{eq:diagonal-mass}
Q_j\ge d_*[\Delta]
\end{equation}
for one \(d_*>0\) independent of \(j\). The uniformity follows from the fixed local lower bound and the irreducibility of \(\Delta\). To justify passage to the limit, take a smooth cutoff equal to one on a fixed neighborhood of the diagonal. All sufficiently small concentration tubes lie in that neighborhood, so the integral against the cutoff is bounded below by the same positive constant. First pass to the weak limit, and then shrink these cutoff neighborhoods to the diagonal. The trace mass of \(\mathbf1_\Delta Q_j\) is thus uniformly positive. Since this restriction is \(d_j[\Delta]\), division by the fixed volume \(\int_\Delta\Theta^m\) gives the uniform \(d_*\). 

\emph{Step 3. Fiber integration and subtraction of diagonal mass.}

Put \(V_j=mB_j^m\) and
\[
R_{j,\eta}=\frac1{V_j}(\pi_1)_*
(\Xi_{j,\eta}^m\wedge\pi_2^*b_j)\in A_j.
\]
The coefficient \(m\) follows from
\((\pi_1)_*(L_j^m\pi_2^*B_j)=mA_jB_j^m\).
We use the decomposition $T^{1,0}_{(x,y)}=T^{1,0}_xY\oplus T_y^{1,0}Y$, and write \[
\Xi_{j,\eta}=\begin{pmatrix}
    H&W\\W^*&G
\end{pmatrix},\qquad \chi_j=\begin{pmatrix}
    b_j&\\&b_j
\end{pmatrix}.
\]
Since $\Xi_{j,\eta}$ is K\"ahler, $G$ is positive definite, and the Schur complement $\widehat{H}:=H-WG^{-1}W^*$ is positive definite. The component of $(\Xi_{j,\eta}+t\pi_2^*b_j)^{m+1}$ having degree $(1,1)$ in the first component of $T^{1,0}_xY$ and degree $(m,m)$ in the second component $T^{1,0}_yY$ is \[
(m+1)\widehat{H}_t\wedge G_t^m, \quad \text{ where }G_t=G+tb_j,\quad \widehat{H}_t=H-WG_t^{-1}W^*.
\]
Differentiating in $t$ gives that the component of $\Xi_{j,\eta}^m\wedge\pi_2^*b_j$ contributing to fiber integration is precisely \[
[(\tr_Gb_j)\widehat{H}+WG^{-1}b_jG^{-1}W^*]\wedge G^m.
\]
Let $E:=\widehat{H}+(\tr_Gb_j)^{-1}WG^{-1}b_jG^{-1}W^*$, then 
\[
E\ge H-WG^{-1}W^*.
\]We claim that $P_{b_j}(E)\le D-\tr_Gb_j:$
This is because by \cite[Lemma 3.5]{Chen21},  \[
D\ge P_{\chi_j}(\Xi_{j,\eta})\ge P_{b_j}(\widehat{H})+\tr_Gb_j.
\]
Since $E\geq \widehat{H}$, we get  $P_{b_j}(E)\leq D-\tr_Gb_j.$

For fixed $x$, the fiber measure \((\tr_Gb_j) G^m/V_j\) has mass one on $\{x\}\times Y$, because
\(\int G^m=B_j^m\) and \(\int(\tr_Gb_j)G^m=mB_j^m=V_j\). Thus the fiber integration gives \[
R_{j,\eta}(x)=\int_{\{x\}\times Y}E(x,y)(\tr_Gb_j)G^m/V_j.
\]
By convexity of the cone function $P_b$, 
\begin{equation}\label{eq:fiber-cone}
P_{b_j}(R_{j,\eta})
\le D-\frac1{V_j}\int(\tr_Gb_j)^2G^m
\le D-\frac{V_j}{B_j^m}=c.
\end{equation}

Choose, at each fixed \(j\), one subsequence with
\[
\Xi_{j,\eta}^m\rightharpoonup Q_j,\qquad
\Xi_{j,\eta}^{m-1}\rightharpoonup U_j,\qquad
\mathbf1_\Delta Q_j=d_j[\Delta],\quad d_j\ge d_*,\quad
\mathbf1_\Delta U_j=0.
\]
The powers have bounded mass by their cohomology classes. The restriction and support statements above give both diagonal identities; further subsequences preserve them. Let \(R_{j,\eta}\rightharpoonup R_j\). We next use eigenvalue truncation to subtract the diagonal contribution while preserving the cone inequality for currents.

\emph{Step 4. Truncation of the large eigenvalues}

For \(0<\zeta<D\), truncate the eigenvalues of \(E\) relative to  \(b_j\) from above at \(1/\zeta\), obtaining \(E^\zeta\)(For instance, at point $(x,y)$, choose a $b_j(x)$-orthonormal basis of $T^{1,0}_xY$ in which $E$ is diagonal, with eigenvalues $0<\lambda_1\le\cdots\le\lambda_m$. Take $E^\zeta=\text{diag}(\min\{\lambda_1,\zeta^{-1}\},...,\min\{\lambda_n,\zeta^{-1}\})$. Such definition is independent of the choice of orthonormal basis). Since \(m\ge2\) and \(P_{b_j}(E)<D\),
\[
D^{-1}b_j\le E^\zeta\le\zeta^{-1}b_j,\quad E^\zeta\le E,\quad
P_{b_j}(E^\zeta)\le P_{b_j}(E)+(m-1)\zeta.
\]
Let \(H_\eta^\zeta\) be its average with the same probability measure above:\[
H_\eta^\zeta(x):=\int_{\{x\}\times Y}E^\zeta(x,y)d\mu_x(y), \qquad \text{ where }d\mu_x(y)=(\tr_Gb_j)G^m/V_j.
\]
Then by convexity of $P_{b_j}$,
\begin{equation}\label{eq:truncation}
D^{-1}b_j\le H_\eta^\zeta\le\zeta^{-1}b_j,\qquad
P_{b_j}(H_\eta^\zeta)\le c+(m-1)\zeta.
\end{equation}
These averages need not be closed, but their coefficients are bounded in \(L^\infty\). For a countable \(\zeta_\ell\downarrow0\), choose one further subsequence on which every \(H_\eta^{\zeta_\ell}\) converges weak-* to \(H^{\zeta_\ell}\). The first inequality in \eqref{eq:truncation} can be passed to the weak-* limit $H^{\zeta_\ell}$ and holds almost everywhere. The second inequality in \eqref{eq:truncation} also passes to the weak-* limits as follows: applying the affine variational formula(follows by completing the square)
\[\begin{aligned}
P_{b_j}(H_\eta^{\zeta_\ell})(x)&=\sup_{\dim_\mathbb{C} L=m-1}\sup_{Z\in \mathbb{C}^{(m-1)\times(m-1)}}\tr\bigl(2\text{Re}(b_j|_L(x)^\frac12Z)-Z^*H_{\eta}^{\zeta_\ell}|_L(x)Z\bigr)\\
&=:\sup_{\dim_\mathbb{C} L=m-1}\sup_ZF_{L,Z,b_j}(H_\eta^{\zeta_\ell}),   
\end{aligned}
\]
gives, for any nonnegative $\phi\in C^\infty_c(Y)$, \[
\int_Y\phi(x)F_{L,Z,b_j}(H^{\zeta_\ell})dV(x)\leq \int_Y(c+(m-1)\zeta_\ell)\phi dV(x).
\]
The integral inequality can be passed to the weak-* limit $H^{\zeta_\ell}$ and therefore, \[
P_{b_j}(H^{\zeta_\ell})\leq c+(m-1)\zeta_\ell
\quad \text{almost everywhere.}\]
This gives \eqref{eq:truncation} almost everywhere for each limit.

Next, we fix one $\zeta_\ell$ and denote it simply by $\zeta$. Let \(\rho_t\) be smooth cutoffs, equal to one near \(\Delta\), with \(0\le\rho_t\le1\) and support shrinking to \(\Delta\). Define the truncated average over the tubular neighborhood by
\[
\mathcal C_{\eta,t}^\zeta(x)
=\int_{\{x\}\times Y}\rho_t E^\zeta d\mu_x(y).
\]
This is a matrix-valued density on the base. Positivity outside the tube gives
\[
R_{j,\eta}-H_\eta^\zeta\ge
\frac1{V_j}(\pi_1)_*(\rho_t\Xi_{j,\eta}^m\pi_2^*b_j)
-\mathcal C_{\eta,t}^\zeta.
\]
As matrix-valued measures on the base, we have
\begin{equation}\label{eq:tube-estimate}
0\le\mathcal C_{\eta,t}^\zeta
\le\frac{m}{V_j\zeta}\,b_j\,(\pi_1)_*
(\rho_t\Xi_{j,\eta}^{m-1}\wedge\pi_2^*b_j).
\end{equation}
 Testing the right side of \eqref{eq:tube-estimate} against a positive smooth \((m-1,m-1)\)-form \(\xi\) on the base gives
\[
\frac{m}{V_j\zeta}\int_{Y\times Y}
\rho_t\Xi_{j,\eta}^{m-1}\pi_2^*b_j\wedge\pi_1^*(b_j\wedge\xi).
\]
First let \(\eta\to0\) with \(j,\zeta,t\) fixed, then let \(t\to0\). The upper bound tends to zero because \(\Xi_{j,\eta}^{m-1}\rightharpoonup U_j, \mathbf1_\Delta U_j=0\). The untruncated term tends to
\[
\frac1{V_j}(\pi_1)_*(\mathbf1_\Delta Q_j\wedge\pi_2^*b_j)
=\frac{d_j}{V_j}b_j.
\]
Consequently, as currents,
\begin{equation}\label{eq:subtraction}
R_j-H^\zeta\ge\frac{d_j}{V_j}b_j.
\end{equation}
Set \(\varepsilon=d_*/(2\sup_jV_j)>0\) and \(S_j:=R_j-\varepsilon b_j\). Then \(S_j\ge H^{\zeta_\ell}+\varepsilon b_j\ge0\) for every \(\ell\). On absolutely continuous parts, order reversal gives
\[
P_{b_j}((S_j)_{\rm ac})\le c+(m-1)\zeta_\ell \quad \text{almost everywhere.}
\]
Intersect the countably many full-measure sets and let \(\ell\to\infty\).

For each constant form $b_0$ with \(0<b_0\leq b_j\)  on a coordinate ball, the same density inequality $P_{b_0}((S_j)_{\rm ac})\leq c+(m-1)\zeta_\ell$ holds. Convexity gives it for convolutions of the absolutely continuous part; adding the convolved positive singular part preserves it. This proves $P_{b_j}(S_j)\leq c+(m-1)\zeta_\ell$ in the local-convolution sense.

Finally, the classes \(A_j-\varepsilon B_j\) stay bounded, so the positive \(S_j\) have bounded mass. A subsequential weak limit \(S\) lies in \(A-\varepsilon B\). Where \(b>0\), every constant positive reference form below \(b\) is below all \(b_j=b+v_j\theta\), so the fixed-convolution inequalities pass to the limit. Shrinking convolutions at Lebesgue points, then letting the constant references approach \(b\) at the center, gives \(P_b(S_{\rm ac})\le c\) almost everywhere on the positive locus of \(b\).
\end{proof}

\section{Strict positivity and regularization}

\begin{samepage}
\begin{lemma}\label{lem:strict-current}
Under the assumptions of Lemma~\ref{lem:concentration} there are fixed \(0<\lambda<1\), \(\gamma,k_0>0\), and a closed current
\[
R\in\lambda A-\gamma[\theta],\qquad R\ge k_0\theta,
\]
whose local regularizations $R_\rho$ with regularization parameter $\rho$, obtained by convolving the entire local potentials, satisfy on a finite atlas
\begin{equation}\label{eq:strict-regularization}
P_b(R_\rho)\le c_2+C'\rho,\qquad R_\rho\ge k'I,
\end{equation}
for fixed \(c_2<c\), \(C',k'>0\). The estimates include points where \(b\) degenerates.
\end{lemma}
\end{samepage}

\begin{proof}
The class \(B\) is nef with \(B^m>0\), so there is a current \(K_B\in B\), \(K_B\ge\kappa\theta\).

Set \(T=S+\varepsilon K_B\in A\). Then \(T\ge\varepsilon\kappa\theta\). Since \(b\le M\theta\),  \((K_B)_{\rm ac}\geq d\,b\)  for a fixed \(d>0\). Jensen's inequality applied to the inverse eigenvalues gives
\begin{equation}\label{eq:strict-margin}
P_b(T_{\rm ac})\le
\frac{c}{1+\varepsilon d\,c/(m-1)}=:c_1<c
\end{equation}
almost everywhere outside the exceptional set.  

Choose \(\lambda<1\) sufficiently close to $1$ and \(\gamma>0\) sufficiently small so that
\[
R=\lambda T-\gamma\theta\ge k_0\theta,\qquad
P_b(R_{\rm ac})\le c_2<c.
\]
This follows from \(R\ge(\lambda-\gamma/(\varepsilon\kappa))T\), homogeneity, and order reversal. Its class is \(\lambda A-\gamma[\theta]\).

On a finite coordinate atlas, identify forms with Hermitian matrices. The density of \(R_{\rm ac}\) is at least \(kI\), and
\(\|b(x)-b(y)\|\le C|x-y|\). Thus
\[
P_{b(y)}(R_{\rm ac}(x))
\le c_2+(m-1)C|x-y|/k, \quad \text{almost everywhere.}
\]
After convolution, convexity and the positivity of the singular part of $R$ give
\begin{equation}\label{eq:convolution}
P_{b(y)}(R_\rho(y))\le c_2+C'\rho,\qquad
R_\rho\ge k'I.
\end{equation}
This argument requires only semipositivity of \(b\), since convexity and order reversal remain valid where \(b\) degenerates.
\end{proof}

Next, we prove a gluing construction analogous to \cite[Lemma4.2]{SSD} allowing $b$ to be semipositive, and obtaining constants  $k, \eta$ uniform in $j,K$.
\begin{samepage}
\begin{lemma}\label{lem:gluing}
Let \(R=r_0+dd^c\varphi\) be fixed and satisfy \eqref{eq:strict-regularization}. Let
\[
Q_j=r_0+H_j,\qquad H_j\longrightarrow H_0>0
\]
smoothly, with \(H_j\) closed and K\"ahler for large \(j\). Suppose smooth \(\psi_j\) give global K\"ahler forms \(Q_j+dd^c\psi_j\) with \(P_b<c\). There exist a fixed \(\delta_0>0\) and a proper analytic set \(E_0=\{\nu(R,\cdot)\ge\delta_0\}\), such that for every compact \(K\subset Y\setminus E_0\) there are smooth K\"ahler forms \(\widehat h_{j,K}\in[Q_j]\) with
\begin{equation}\label{eq:glued-forms}
\begin{aligned}
P_b(\widehat h_{j,K})&<c &&\text{on }Y,\\
\widehat h_{j,K}\ge k\theta,\qquad
P_b(\widehat h_{j,K})&\le c-\eta &&\text{on }K,
\end{aligned}
\end{equation}
where \(k,\eta>0\) are independent of \(j,K\). The construction applies for every sufficiently large \(j\). 
\end{lemma}
\end{samepage}

\begin{proof}
Throughout this proof, \(r>0\) denotes the fixed chart radius and \(\rho>0\) the variable smoothing radius. 

Take finitely many holomorphic charts containing \(B_{2r}(x_i)\), normalized so that \(H_0\) is Euclidean at the centers, with \(B_r(x_i)\) covering \(Y\).  For convenience, we also set 
\[
\begin{aligned}
\mathcal I(x)&=\{i: x\in \overline{B_{9r/5}(x_i)}\},    \\
\mathcal I_{\mathrm{in}}(x)
&=\{i\in\mathcal I(x):x\in B_r^i(x_i)\},\\
\mathcal I_{\mathrm{out}}(x)
&=\{i\in\mathcal I(x):x\notin B_{8r/5}^i(x_i)\}.
\end{aligned}
\]
By starting with larger charts, shrinking the common \(r\), and using uniform continuity of the metric and coordinate derivatives, arrange the half-radius overlap inclusion
\begin{equation}\label{eq:overlap}
B^i_{\rho/2}(x)\subset B^k_\rho(x),\qquad \rho<r/20
\end{equation}for every $x\in Y$ and $i,k\in \mathcal{I}(x)$.

Let \(r_i\) be fixed local potentials of \(r_0\). Remove constant, linear, and pure holomorphic quadratic Taylor terms from local potentials \(h_{j,i}\) of \(H_j\). Uniform smooth convergence, decreasing \(r\), and discarding finitely many \(j\) give
\begin{equation}\label{eq:quadratic}
|h_{j,i}(z)-|z_i|^2|\le\epsilon r^2
\end{equation}
on these charts, for a fixed small \(\epsilon\). The potentials are uniformly bounded; the \(r_i\) have uniformly bounded first derivatives. Set
\[
\Phi_i=\varphi+r_i,\quad q_{j,i}=r_i+h_{j,i},\quad
v_{j,i,\rho}=\Phi_i*\mathcal{K}_\rho-q_{j,i},
\]where $\mathcal{K}_\rho(z)=\rho^{-2m}\mathcal{K}(z/\rho)$ is the rescaled smooth nonnegative radial kernel $\mathcal{K}\in C^\infty_c(B_1(0))$, with $\int_{\mathbb{C}^m}\mathcal{K}(z)dz=1$.

$\Phi_i$ is plurisubharmonic, and exactly
\begin{equation}\label{eq:local-class}
Q_j+dd^c v_{j,i,\rho}=dd^c(\Phi_i*\mathcal{K}_\rho)=R_\rho^i.
\end{equation}
Thus $Q_j+dd^cv_{j,i,\rho}$ satisfies the cone condition with the fixed margins \eqref{eq:strict-regularization}.

Set  \(R_*=r/16\) and
\[
\widehat{\Phi}_{i,\rho}(x):=\sup_{y\in B_\rho^i(x)}\Phi_i(y),\qquad
\nu_i(x,\rho)=\frac{\widehat{\Phi}_{i,R_*}(x)-\widehat{\Phi}_{i,\rho}(x)}
{\log R_*-\log\rho}.
\]
The slopes $\nu_i(x,\rho)$ decrease to the Lelong number as \(\rho\downarrow0\). \cite[Lemma 4.2]{Chen21} gives
\begin{equation}\label{eq:radial-estimates}
\begin{aligned}
0&\le \widehat{\Phi}_{i,\rho}(x)-\widehat{\Phi}_{i,\rho/a}(x)\le\nu_i(x,\rho)\log a,\\
0&\le \widehat{\Phi}_{i,\rho}(x)-\Phi_{i}*\mathcal{K}_\rho(x)\le C_0\nu_i(x,\rho).
\end{aligned}
\end{equation}
Here \cite[Lemma 4.2]{C21} applies for \(a\ge1\), and \(C_0\) depends only on the kernel and dimension. 

For any fixed positive radius \(\rho\), \(\widehat{\Phi}_{i,\rho}(x)\) is bounded above(by the local boundedness above by the property of plurisubharmonic functions) and below (by the local $L^1$ bound) on compact subcharts. And $\widehat{\Phi}_{i,\rho}(x)$ is continuous. Thus fixing $\rho$,  $\nu_i(x,\rho)$ is uniformly bounded, and monotonicity gives
\begin{equation}\label{eq:slope-bound}
0\le\nu_i(x,\rho)\le C_\nu
\end{equation}
for $x\in Y,i\in \mathcal{I}(x)$ and $0<\rho<\frac{r}{20}$.

For any fixed \(\delta_0>0\) and compact
\(F\subset Y\setminus\{x:\nu(R,x)\ge\delta_0\}\), there is a radius \(\rho_F>0\) such that
\[
\nu_i(x,\rho)<2\delta_0
\quad\text{for }x\in F,\ i\in\mathcal I(x),\ 0<\rho<\rho_F.
\]
Indeed, for each \(i\), the set \(F_i=F\cap\overline{B_{9r/5}(x_i)}\) is compact. At every \(x\in F_i\), pointwise convergence gives a radius at which \(\nu_i(x,\rho)<3\delta_0/2\). Continuity at that radius gives a neighborhood where the slope is below \(2\delta_0\), and monotonicity preserves this bound at all smaller radii. A finite subcover of each \(F_i\), followed by the minimum over the finite atlas, gives \(\rho_F\). Continuity of the limiting Lelong function is not required.

If \(x\in (B^i_{9r/5}\setminus B^i_{8r/5})\cap B^k_r\), the quadratic normalization gives, for every $j\ge j_0$,
\[
h_{j,i}(x)-h_{j,k}(x)
\ge((8/5)^2-1-2\epsilon)r^2\ge d_0r^2>0.
\]
Inequalities in \eqref{eq:radial-estimates} and the bounded first derivatives of \(r_i\) give
\[
\begin{aligned}
v_{j,i,\rho}(x)
&\le \widehat{\Phi}_{i,\rho/2}(x)+\nu_i(x,\rho)\log2-r_i(x)-h_{j,i}(x)\\
&\le\sup_{B^i_{\rho/2}(x)}\varphi-h_{j,i}(x)+\nu_i(x,\rho)\log2+O(\rho),\\
\sup_{B^k_\rho(x)}\varphi
&\le \widehat{\Phi}_{k,\rho}(x)-r_k(x)+O(\rho)\\
&\le v_{j,k,\rho}(x)+h_{j,k}(x)+C_0\nu_k(x,\rho)+O(\rho).
\end{aligned}
\]
Using \eqref{eq:overlap}, we get 
\[
v_{j,i,\rho}(x)-v_{j,k,\rho}(x)
\le-d_0r^2+C_1(\nu_i(x,\rho)+\nu_k(x,\rho))+C_r\rho.
\]
Choose \(\delta_0>0\) with \(8C_1\delta_0<d_0r^2/4\), and choose \(0<\rho_{\mathrm{sep}}<r/20\) such that \(C_r\rho_{\mathrm{sep}}\le d_0r^2/4\). It follows that, for every \(j\ge j_0\), \(x\in Y\), and \(0<\rho<\rho_{\mathrm{sep}}\),
\begin{equation}\label{eq:branch-comparison}
\begin{gathered}
\max_{\ell\in\mathcal I(x)}\nu_\ell(x,\rho)\le4\delta_0\\
\Longrightarrow\quad
v_{j,i,\rho}(x)\le v_{j,k,\rho}(x)-d_0r^2/2
\quad
\begin{matrix}
\text{for every }i\in\mathcal I_{\mathrm{out}}(x),\\
\text{and every }k\in\mathcal I_{\mathrm{in}}(x).
\end{matrix}
\end{gathered}
\end{equation}

Fix compact \(K\), choose open sets $O,O'$ with \(E_0\subset O'\Subset O\), such that  \(\overline O\cap K=\varnothing\).  Fix $
0<e_0<\min(1,d_0r^2/16)$. Set
\[
w_{j,\rho}:=\psi_j+3\delta_0\log\rho\text{ on }O.
\]
We claim that for each \(j\ge j_0\) there is a radius \(\rho_j\in(0,r/20)\) such that both implications below hold for every \(x\in Y\) and every \(0<\rho<\rho_j\). The argument follows \cite[Proposition 4.1(1)--(2)]{Chen21}

\emph{Implication 1:}
If \(\nu_i(x,\rho)\le2\delta_0\) for every \(i\in\mathcal I(x)\), the definition $\nu_i(x,\rho)$, the fixed-radius lower bound of $\widehat{\Phi}_{i,R_*}$, \eqref{eq:radial-estimates}, and boundedness of \(q_{j,i}\) give
\[\begin{aligned}
v_{j,i,\rho}(x)=&\Phi_i*\mathcal{K_\rho}(x)-q_{j,i}\\
\ge&\widehat{\Phi}_{i,\rho}(x)-C_0\nu_{i}(x,\rho)(x)-q_{j,i}\\
=&\widehat{\Phi}_{i,R_*}(x)-\nu_i(x,\rho)(\log R_*-\log\rho)-C_0\nu_i(x,\rho)-q_{j,i}\\
\ge&
2\delta_0\log\rho-C.
\end{aligned}
\]The constant \(C\) is independent of \(j\).

Consequently
\begin{equation}\label{eq:small-slopes}
\max_{i\in \mathcal I(x)}\nu_i(x,\rho)\le2\delta_0
\ \Longrightarrow\
\max_{i\in \mathcal{I}(x)} v_{j,i,\rho}(x)>\sup_O\psi_j+3\delta_0\log\rho+e_0.
\end{equation}
 The finite \(\sup_O\psi_j\) affects only the choice of radius $\rho$.

\emph{Implication 2:}
\begin{equation}\label{eq:active-slopes}
\begin{gathered}
\max_{i\in\mathcal I(x)}v_{j,i,\rho}(x)
\ge\inf_O\psi_j+3\delta_0\log\rho-e_0
\Longrightarrow
\max_{k\in\mathcal I(x)}\nu_k(x,\rho)<4\delta_0.
\end{gathered}
\end{equation}

Fix an index \(i_0\) satisfying the hypothesis. 
Since
\(\widehat{\Phi}_{i_0,\rho}\ge\Phi_{i_0}*\mathcal{K}_\rho=v_{j,i_0,\rho}+q_{j,i_0}\),
\[
\widehat{\Phi}_{i_0,\rho}(x)
\ge3\delta_0\log\rho-C_j,
\]
where \(C_j\) may depend on \(\inf_O\psi_j\) but is independent of \(x,i_0,\rho\). Applying \eqref{eq:radial-estimates} with \(a=2\) and then \eqref{eq:slope-bound} yields
\[
\widehat{\Phi}_{i_0,\rho/2}(x)
\ge3\delta_0\log\rho-C_j-C_\nu\log2.
\]
Now fix any \(k\in\mathcal I(x)\). By \eqref{eq:overlap},
\(B_{\rho/2}^{i_0}(x)\subset B_\rho^k(x)\). The fixed smooth potentials are bounded on \(\overline{B_{19r/10}^i(x_i)}\); hence a constant \(C_{\mathrm{pot}}\), independent of \(x,i_0,k,\rho\), bounds \(|r_k-r_{i_0}|\) on the smaller ball. Consequently,
\[
\begin{aligned}
\widehat{\Phi}_{k,\rho}(x)
&\ge\sup_{B_{\rho/2}^{i_0}(x)}
\bigl(\Phi_{i_0}+r_k-r_{i_0}\bigr)\\
&\ge\widehat{\Phi}_{i_0,\rho/2}(x)-C_{\mathrm{pot}}\\
&\ge3\delta_0\log\rho-C_j'.
\end{aligned}
\]
Together with the fixed-radius upper bound of  $\widehat{\Phi}_{k,R_*}(x)$, this gives
\[
\nu_k(x,\rho)
\le\frac{C_j''-3\delta_0\log\rho}{\log R_*-\log\rho}.
\]
The right-hand side tends to \(3\delta_0\) as \(\rho\downarrow0\) and is independent of \(x,k\). It is therefore below \(4\delta_0\) for all sufficiently small \(\rho\), proving the implication. 
Taking the smaller of the radii obtained for the two implications gives \(\rho_j\).

Choose 
\[
0<\rho=\rho(j,K)<
\min\{\rho_j,\rho_{\mathrm{sep}},\rho_{Y\setminus O'}\}.
\]
The last radius $\rho_{Y\setminus O'}$ is supplied by $\rho_F$ discussed before for $F=Y\setminus O'$.
Thus
\(\nu_i(x,\rho)<2\delta_0\) for every \(x\in Y\setminus O'\) and \(i\in\mathcal I(x)\).
Set 
\[
M_{j,\rho}(x)=\max_{i\in\mathcal I(x)}v_{j,i,\rho}(x).
\]
On \(O\setminus\overline{O'}\), \eqref{eq:small-slopes} gives
\(M_{j,\rho}(x)>w_{j,\rho}(x)+e_0\). Since \(O'\Subset O\), this separates the comparison potential from the local maximum throughout a neighborhood of \(\partial O\) within \(O\).

For a local potential with \(i\in\mathcal I_{\mathrm{out}}(x)\), choose \(k\in\mathcal I_{\mathrm{in}}(x)\). If \(x\notin O\), the slope bound on \(Y\setminus O'\) and \eqref{eq:branch-comparison} give
\(v_{j,i,\rho}(x)\le v_{j,k,\rho}(x)-d_0r^2/2\).
If \(x\in O\) and
\[
M_{j,\rho}(x)<\inf_O\psi_j+3\delta_0\log\rho-e_0,
\]
then \(w_{j,\rho}(x)>M_{j,\rho}(x)+e_0\). In the remaining case, \eqref{eq:active-slopes} and \eqref{eq:branch-comparison} again give
\(v_{j,i,\rho}(x)\le v_{j,k,\rho}(x)-d_0r^2/2\).
Hence, for every \(i\in\mathcal I_{\mathrm{out}}(x)\), the value \(v_{j,i,\rho}(x)\) lies below the maximum of the local potentials, together with \(w_{j,\rho}(x)\) when \(x\in O\), by at least \(\min\{e_0,d_0r^2/2\}\). The comparison potential lies below \(M_{j,\rho}\) by more than \(e_0\) on \(O\setminus\overline{O'}\). These estimates allow the removal of each potential near the boundary of its domain. Figure~\ref{fig:gluing-comparisons} summarizes the comparisons.

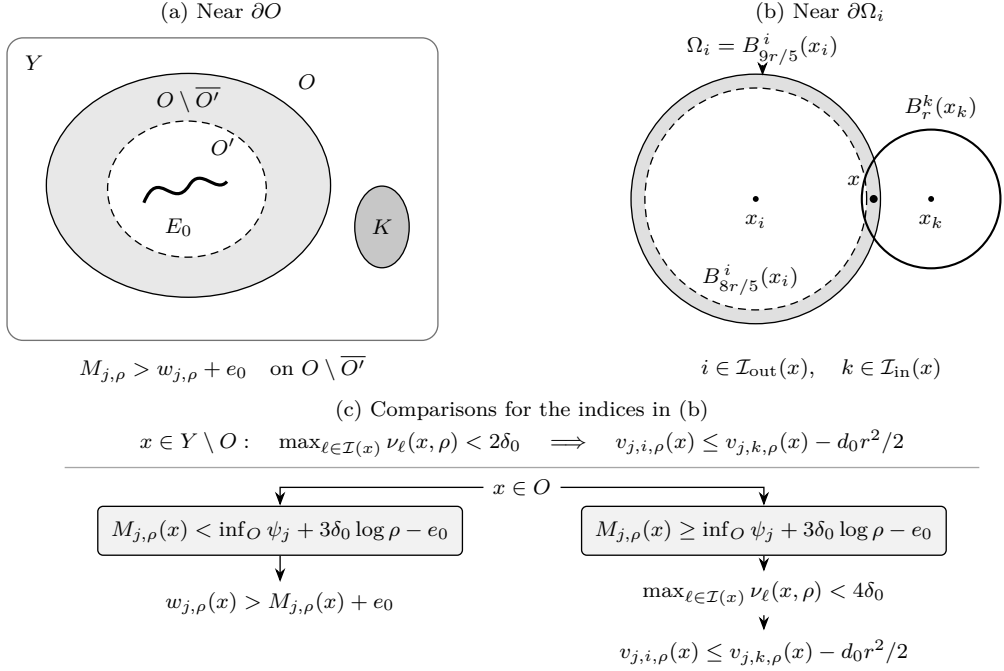
\begin{figure}[!htbp]
\centering
\begin{minipage}[t]{0.48\linewidth}
\centering\footnotesize
\textup{(a) Near \(\partial O\)}\par\smallskip
\begin{tikzpicture}[x=1cm,y=1cm,font=\footnotesize,
  line width=0.5pt,>=Stealth]
\path[use as bounding box] (0,0) rectangle (5.8,4.7);
\draw[rounded corners=6pt,black!55] (0.05,0.60) rectangle (5.75,4.60);
\node[anchor=north west] at (0.18,4.48) {\(Y\)};
\draw[fill=black!9] (2.45,2.65) ellipse (1.88 and 1.48);
\draw[fill=white,densely dashed] (2.43,2.60) ellipse (1.05 and 0.90);
\node at (4.02,4.02) {\(O\)};
\node at (2.45,3.79) {\(O\setminus\overline{O'}\)};
\node at (2.91,3.16) {\(O'\)};
\draw[very thick] (1.86,2.42)
  .. controls (2.06,2.91) and (2.24,2.31) .. (2.46,2.65)
  .. controls (2.62,2.91) and (2.81,2.53) .. (2.96,2.70);
\node at (2.32,2.10) {\(E_0\)};
\draw[fill=black!22] (5.01,2.10) ellipse (0.36 and 0.54);
\node at (5.01,2.10) {\(K\)};
\node at (2.90,0.23) {\(M_{j,\rho}>w_{j,\rho}+e_0\quad\text{on }O\setminus\overline{O'}\)};
\end{tikzpicture}
\end{minipage}\hfill
\begin{minipage}[t]{0.48\linewidth}
\centering\footnotesize
\textup{(b) Near \(\partial\Omega_i\)}\par\smallskip
\begin{tikzpicture}[x=1cm,y=1cm,font=\footnotesize,
  line width=0.5pt,>=Stealth]
\path[use as bounding box] (0,0) rectangle (5.8,4.7);
\draw[fill=black!12] (2.03,2.47) circle (1.65);
\draw[fill=white,densely dashed] (2.03,2.47) circle (1.467);
\draw[thick] (4.35,2.47) circle (0.917);
\fill (2.03,2.47) circle (1.1pt);
\node[below=3pt] at (2.03,2.47) {\(x_i\)};
\fill (4.35,2.47) circle (1.1pt);
\node[below=3pt] at (4.35,2.47) {\(x_k\)};
\fill (3.59,2.47) circle (1.5pt);
\node[above left=2pt] at (3.59,2.47) {\(x\)};
\node at (2.12,4.45) {\(\Omega_i=B_{9r/5}^{\,i}(x_i)\)};
\draw[->] (2.12,4.24) -- (2.12,4.10);
\node at (4.48,3.69) {\(B_r^k(x_k)\)};
\node at (1.96,1.40) {\(B_{8r/5}^{\,i}(x_i)\)};
\node at (2.90,0.23) {\(i\in\mathcal I_{\mathrm{out}}(x),\quad k\in\mathcal I_{\mathrm{in}}(x)\)};
\end{tikzpicture}
\end{minipage}

\medskip
{\footnotesize\textup{(c) Comparisons for the indices in (b)}}\par\smallskip
\begin{tikzpicture}[x=1cm,y=0.82cm,font=\footnotesize,
  line width=0.5pt,>=Stealth,
  condition/.style={draw,rounded corners=2pt,fill=black!5,
    inner sep=5pt,minimum height=0.53cm}]
\path[use as bounding box] (0,-0.12) rectangle (12.5,3.70);
\node at (6.25,3.47) {\(x\in Y\setminus O:\quad
  \max_{\ell\in\mathcal I(x)}\nu_\ell(x,\rho)<2\delta_0
  \quad\Longrightarrow\quad
  v_{j,i,\rho}(x)\le v_{j,k,\rho}(x)-d_0r^2/2\)};
\draw[black!35] (0.20,3.08) -- (12.30,3.08);
\node (inside) at (6.25,2.76) {\(x\in O\)};
\node[condition] (lower) at (3.05,2.06)
  {\(M_{j,\rho}(x)<\inf_O\psi_j+3\delta_0\log\rho-e_0\)};
\node[condition] (upper) at (9.45,2.06)
  {\(M_{j,\rho}(x)\ge\inf_O\psi_j+3\delta_0\log\rho-e_0\)};
\draw[->] (inside.west) -| (lower.north);
\draw[->] (inside.east) -| (upper.north);
\node (comparison) at (3.05,0.90)
  {\(w_{j,\rho}(x)>M_{j,\rho}(x)+e_0\)};
\draw[->] (lower.south) -- (comparison.north);
\node (slopes) at (9.45,1.05)
  {\(\max_{\ell\in\mathcal I(x)}\nu_\ell(x,\rho)<4\delta_0\)};
\draw[->] (upper.south) -- (slopes.north);
\node (local) at (9.45,0.10)
  {\(v_{j,i,\rho}(x)\le v_{j,k,\rho}(x)-d_0r^2/2\)};
\draw[->] (slopes.south) -- (local.north);
\end{tikzpicture}
\caption{Comparisons in Lemma~\ref{lem:gluing} for the chosen \(j,K,\rho\).
The positive gaps in (a) and (c) allow \(w_{j,\rho}\) and
\(v_{j,i,\rho}\) to be omitted from the regularized maximum near
\(\partial O\) and \(\partial\Omega_i\), respectively.
The shapes indicate only set inclusions, with no boundary regularity assumed.}
\label{fig:gluing-comparisons}
\end{figure}

\Needspace{8\baselineskip}
Decrease \(\rho=\rho(j,K)\) so that \(C'\rho\le(c-c_2)/2\), and set \(\eta=(c-c_2)/2\). Let \(\mathrm M_\tau\) denote the standard regularized maximum of \cite[I, Lemma 5.18]{Demailly}, with all smoothing parameters equal to
\(0<\tau<\frac14\min\{e_0,d_0r^2/2\}\). Define
\[
\varphi_{j,K}(x)=
\begin{cases}
\mathrm M_\tau\bigl((v_{j,i,\rho}(x))_{i\in\mathcal I(x)},
 w_{j,\rho}(x)\bigr),&x\in O,\\[2pt]
\mathrm M_\tau\bigl((v_{j,i,\rho}(x))_{i\in\mathcal I(x)}\bigr),
 &x\in Y\setminus O.
\end{cases}
\]
The preceding gaps and the omission property in \cite[I, Lemma 5.18(c)]{Demailly} imply \(\varphi_{j,K}\in C^\infty(Y)\). When \(E_0=\varnothing\), use the second formula on all of \(Y\), with \(0<\tau<d_0r^2/8\).

Writing \(f_\ell\) for the local arguments of \(\mathrm M_\tau\) and \(a_\ell=\partial_\ell\mathrm M_\tau(f)\), its standard properties give
\[
\begin{gathered}
\widehat h_{j,K}:=Q_j+dd^c\varphi_{j,K}
\ge\sum_\ell a_\ell\bigl(Q_j+dd^c f_\ell\bigr)>0,\\
a_\ell\ge0,\qquad\sum_\ell a_\ell=1.
\end{gathered}
\]
Thus \(\widehat h_{j,K}\in[Q_j]\). On \(K\), only the local convolution potentials occur. Lemma~\ref{lem:matrix-cone} and \eqref{eq:strict-regularization} therefore yield
\[
\begin{aligned}
P_b(\widehat h_{j,K})&<c&&\text{on }Y,\\
\widehat h_{j,K}&\ge k\theta,\qquad
P_b(\widehat h_{j,K})\le c-\eta&&\text{on }K,
\end{aligned}
\]
where \(k>0\) is fixed by comparison of the coordinate metrics with \(\theta\). Both \(k\) and \(\eta\) are independent of \(j,K\), proving \eqref{eq:glued-forms}.
\end{proof}

\section{Finiteness of null subvarieties}
The following lemma combines the dimension-reduction strategy of Khalid--Sj\"ostr\"om Dyrefelt \cite[proofs of Theorems 3.1 and 3.9]{KSD26} with the analytic-containment argument of Sivaram--Sj\"ostr\"om Dyrefelt \cite[Proposition 4.3 and its proof, equations (4.10)--(4.11)]{SSD26}. We carry out the latter construction on resolutions of arbitrary analytic subvarieties, to obtain the quantitative estimate used in the induction.
\begin{lemma}[A uniform estimate outside an analytic subset]\label{lem:local-gap}
Assume the pair $(\alpha,\beta)$ is $J$-nef, and let \(Z\subseteq X\) be an irreducible analytic subvariety of dimension \(m\ge2\), allowing \(Z=X\). There is a proper analytic subset \(E_Z\subsetneq Z\) and, for each \(1\le p<m\), a constant \(\delta_{Z,p}>0\) such that every proper positive-dimensional subvariety \(V\subset Z\) not contained in \(E_Z\) satisfies
\begin{equation}\label{eq:local-gap}
J(V,\alpha,\beta)\ge\delta_{Z,p}\int_V\beta^p.
\end{equation}
In particular, every proper positive-dimensional subvariety of $Z$ that is null with respect to the ambient constant $c=\cab$ is contained in \(E_Z\).
\end{lemma}

\begin{proof}
Resolve \(Z\) by an embedded resolution. Write \(g:Y\to Z\), \(a=g^*a_X\), \(b=g^*b_X\), \(A=g^*\alpha\), \(B=g^*\beta\), and choose a K\"ahler form \(\theta\) on \(Y\). The top inequality is \(cA^m-mBA^{m-1}=J(Z,\alpha,\beta)\ge0\) if \(Z\ne X\), and equals zero if \(Z=X\). Apply Lemmas \ref{lem:resolution-approximation}, \ref{lem:concentration} and \ref{lem:strict-current} with this same ambient \(c\). For the gluing lemma choose
\[
\begin{gathered}
r_0=\lambda a-\gamma\theta,\qquad Q_j=t_ja+u_j\theta,\\
H_j=(t_j-\lambda)a+(u_j+\gamma)\theta
\longrightarrow(1-\lambda)a+\gamma\theta>0.
\end{gathered}
\]
The global comparison forms are \(h_j\) from \eqref{eq:resolved-cone}, and \(P_b(h_j)\le P_{b_j}(h_j)<c\). The \(dd^c\)-lemma supplies their smooth potentials relative to \(Q_j\). Every hypothesis of Lemma~\ref{lem:gluing} holds, with one fixed \(E_0\) and constants independent of \(j,K\).

Let \(E_0=E_{\delta_0}(R)\) and define
\[
E_Z=g(E_0)\cup\operatorname{Sing}Z.
\]
 Since \(Z\) is irreducible its singular locus is proper, so \(E_Z\) is a proper analytic subset of $Z$.

Take an irreducible \(p\)-fold \(V\subset Z\) not contained in \(E_Z\). In particular \(V\not\subset\operatorname{Sing}Z\), so its strict transform \(\widetilde V\) is irreducible and maps birationally onto \(V\). It is not contained in \(E_0\). For each fixed \(K\Subset Y\setminus E_0\), \eqref{eq:glued-forms} gives
\[
c\widehat h_{j,K}^p-pb\widehat h_{j,K}^{p-1}\ge0
\quad\hbox{on }Y,
\]
and on \(K\) the same form is at least \(\eta k^p\theta^p\). Indeed, extend every \(p\)-plane to an \((m-1)\)-plane and use \(b\ge0\) to bound its trace by \(c-\eta\). The constants \(k,\eta\) here do not depend on \(K\) or \(j\).

Integration gives
\[
I_j(V):=\int_{\widetilde V}(cA_j^p-pBA_j^{p-1})
\ge\eta k^p\int_{\widetilde V_{\rm reg}\cap K}\theta^p.
\]
The left side is the cohomological value and is independent of \(K\), although the chosen representative depends on \(K\). First let \(j\to\infty\) with \(V,K\) fixed:
\[
J(V,\alpha,\beta)\ge\eta k^p\int_{\widetilde V_{\rm reg}\cap K}\theta^p.
\]
Now exhaust \(Y\setminus E_0\) by compact sets. The proper analytic intersection \(\widetilde V\cap E_0\), and also \(\widetilde V_{\rm sing}\), have zero smooth \(p\)-dimensional volume. Monotone convergence yields
\[
J(V,\alpha,\beta)\ge\eta k^p\int_{\widetilde V}\theta^p.
\]
Choose \(M_Z>0\) such that \(b\le M_Z\theta\). Projection gives
\[
\int_{\widetilde V}\theta^p
\ge M_Z^{-p}\int_{\widetilde V}b^p
=M_Z^{-p}\int_V\beta^p.
\]
Thus \eqref{eq:local-gap} holds with \(\delta_{Z,p}=\eta(k/M_Z)^p>0\).

\end{proof}

\begin{proof}[Proof of Theorem~\ref{thm:null-finiteness}]
For each irreducible analytic \(Z\subseteq X\) of dimension at least two
that occurs below, fix the proper analytic subset \(E_Z\) and the constants
\(\delta_{Z,p}>0\), \(1\le p<\dim Z\), furnished by
Lemma~\ref{lem:local-gap}. The same choice is used whenever \(Z\) occurs
again. Write \(\operatorname{Irr}_+(A)\) for the family of
positive-dimensional irreducible components of a compact analytic set
\(A\). This family is finite by compactness and local finiteness of the
irreducible decomposition. Define inductively
\[
\mathcal F_0=\{X\},\qquad
\mathcal F_{q+1}
=\bigcup_{\substack{Z\in\mathcal F_q\\ \dim Z\ge2}}
\operatorname{Irr}_+(E_Z),\qquad q\ge0.
\]
Since \(E_Z\subsetneq Z\) and \(Z\) is irreducible, every
\(W\in\operatorname{Irr}_+(E_Z)\) satisfies
\(1\le\dim W<\dim Z\). Induction on \(q\) gives
\[
\#\mathcal F_q<\infty,\qquad
1\le\dim Z\le n-q\quad(Z\in\mathcal F_q).
\]
In particular, \(\mathcal F_n=\varnothing\), and
\[
\mathcal F:=\bigcup_{q=0}^{n-1}\mathcal F_q
\]
is finite.  For every \(Z\in\mathcal F\) with \(\dim Z\ge2\),
\(\operatorname{Irr}_+(E_Z)\subseteq\mathcal F\); moreover every
\(Z\in\mathcal F\setminus\{X\}\) is contained in \(E_X\). The construction includes singular \(Z\) and does not require \(Z\) to
be null. 

Let \(V\subsetneq X\) be a reduced irreducible analytic subvariety of
dimension \(p\ge1\). Since \(X\in\mathcal F\), there is
\(Z_V\in\mathcal F\) such that
\[
V\subseteq Z_V,\qquad
\dim Z_V=\min\{\dim Z:Z\in\mathcal F,\ V\subseteq Z\}.
\]
Suppose \(V\subsetneq Z_V\). Then \(p<\dim Z_V\). If
\(V\subseteq E_{Z_V}\), irreducibility of \(V\) and the finite
irreducible decomposition of \(E_{Z_V}\) give
\[
V\subseteq W\quad\text{for some }
W\in\operatorname{Irr}_+(E_{Z_V})\subseteq\mathcal F.
\]
But \(\dim W<\dim Z_V\), contradicting the choice of \(Z_V\).
Thus \(V\not\subseteq E_{Z_V}\), and Lemma~\ref{lem:local-gap} yields
\[
V\subsetneq Z_V\quad\Longrightarrow\quad
\frac{J(V,\alpha,\beta)}{\int_V\beta^p}\ge\delta_{Z_V,p}>0.
\]
Here \(\int_V\beta^p>0\) because \(\beta\) is K\"ahler.

Let \(\mathcal N\) denote the collection of proper positive-dimensional
null subvarieties. The preceding implication shows that every
\(V\in\mathcal N\) satisfies \(V=Z_V\). Consequently,
\[
\mathcal N
=\{Z\in\mathcal F\setminus\{X\}:J(Z,\alpha,\beta)=0\},\qquad
\#\mathcal N\le\#\mathcal F-1<\infty.
\]
Its union is analytic, being a finite union of analytic subsets, and is
proper since
\[
\bigcup_{V\in\mathcal N}V\subseteq E_X\subsetneq X.
\]

For the uniform estimate, set
\[
\begin{aligned}
\mathcal D_1
&=\{\delta_{Z,p}:Z\in\mathcal F,\ 1\le p<\dim Z\},\\
\mathcal D_2
&=\left\{\frac{J(Z,\alpha,\beta)}{\int_Z\beta^{\dim Z}}:
 Z\in\mathcal F\setminus(\{X\}\cup\mathcal N)\right\}.
\end{aligned}
\]
Both sets are finite subsets of \((0,\infty)\), and
\(\delta_{X,1}\in\mathcal D_1\). Hence
\[
\delta:=\min(\mathcal D_1\cup\mathcal D_2)>0.
\]
For a non-null \(V\), the two possibilities for \(Z_V\) give
\[
\begin{aligned}
V\subsetneq Z_V
&\quad\Longrightarrow\quad
\frac{J(V,\alpha,\beta)}{\int_V\beta^p}\ge\delta_{Z_V,p}\ge\delta,\\
V=Z_V
&\quad\Longrightarrow\quad
\frac{J(V,\alpha,\beta)}{\int_V\beta^p}\in\mathcal D_2\subseteq[\delta,\infty).
\end{aligned}
\]
This proves the uniform estimate. Finally,  \(J\)-nefness gives,
for every proper reduced irreducible \(p\)-fold \(V\),
\[
\frac{J(V,\alpha,\beta)}{\int_V\beta^p}<\delta
\quad\Longleftrightarrow\quad J(V,\alpha,\beta)=0,
\]
as asserted.

\end{proof}

\section{Optimal destabilizers}
We recall the class-shift reduction from Sj\"ostr\"om Dyrefelt \cite[Lemma 14]{SD20}, Khalid--Sj\"ostr\"om Dyrefelt \cite[proof of Theorem 3.10 and Remark 3.11]{KSD26}, and Sivaram--Sj\"ostr\"om Dyrefelt \cite[Section 2.2.1]{SSD26}. We include the calculation for completeness.
\begin{samepage}
\begin{lemma}\label{lem:shift} If \(\Gamma\le0\), set
\[
\beta_0=\beta-\Gamma\alpha,\qquad c_0=c-n\Gamma.
\]
Then \(\beta_0\) is K\"ahler, \(c_0=n\beta_0\alpha^{n-1}/\alpha^n>0\), the pair \((\alpha,\beta_0)\) is  \(J\)-nef, and its null subvarieties are exactly the original optimal destabilizers.
\end{lemma}
\end{samepage}

\begin{proof}
Choose \(M>0\) with \(b_X\le M a_X\). For every \(p\)-fold, \(0<\mu(V,\alpha,\beta)\le pM\), so \(r(V,\alpha,\beta)\ge(c-pM)/(n-p)\). The minimum $r(V,\alpha,\beta)$ is finite. 

For \(\Gamma\le0\), the form \(b_0=b_X-\Gamma a_X\) is K\"ahler. Linearity gives the asserted formula for \(c_0\), and for each proper subvariety it gives
\begin{equation}\label{eq:shift}
\begin{aligned}
J(V,\alpha,\beta_0)
&=\int_V\bigl((c-n\Gamma)\alpha^p-p(\beta-\Gamma\alpha)\alpha^{p-1}\bigr)\\
&=J(V,\alpha,\beta)-(n-p)\Gamma\int_V\alpha^p,\\
\frac{J(V,\alpha,\beta_0)}{(n-p)\int_V\alpha^p}&=r(V,\alpha,\beta)-\Gamma\ge0.
\end{aligned}
\end{equation}
Taking infima gives zero. Equality in \eqref{eq:shift} is exactly original optimality. 
\end{proof}

The following attainment is also part of Sj\"ostr\"om Dyrefelt's numerical characterization of the stability threshold \cite[Theorems 5 and 24]{SD20}, conditional there on the Lejmi--Sz\'ekelyhidi criterion. In the present nonpositive-threshold setting, that criterion is now available by \cite{S20}; we include the short argument for completeness.

\begin{lemma}[Attainment]\label{lem:attainment}
If the threshold \(\Gamma\leq 0\), it is attained by a proper positive-dimensional irreducible subvariety.
\end{lemma}

\begin{proof}
Apply Lemma~\ref{lem:shift}. For the shifted pair the threshold is zero, and the top \(J\)-intersection is zero by the definition \(c_0=n\beta_0\alpha^{n-1}/\alpha^n\). If no proper null subvariety existed, every proper numerical inequality would be strict. Apply the strict numerical cone criterion stated above with \(C=c_0\alpha\) and \(B=\beta_0\). It gives a K\"ahler form \(a_0\in\alpha\) with \(P_{b_0}(a_0)<c_0\), where \(b_0\in\beta_0\) is fixed and smooth. Compactness gives
\[
\delta=c_0-\sup_XP_{b_0}(a_0)>0.
\]
Restriction to every \(p\)-plane, \(1\le p<n\), yields
\[
c_0a_0^p-pb_0a_0^{p-1}\ge\delta a_0^p.
\]
Integrating over each proper \(p\)-fold \(V\), including a singular one by its regular locus and its integration current, gives
\[
\frac{J(V,\alpha,\beta_0)}{(n-p)\int_V\alpha^p}
\ge\frac{\delta}{n-p}\ge\frac{\delta}{n-1}.
\]
This contradicts the shifted threshold zero. A shifted null subvariety therefore exists, and it is an original optimizer by the identity \eqref{eq:shift}.
\end{proof}

\begin{proof}[Proof of Theorem~\ref{thm:optimal-finiteness}]
Lemma~\ref{lem:shift} shows that \(\Gamma\) is finite; Lemma~\ref{lem:attainment} supplies an optimal destabilizer. Theorem~\ref{thm:null-finiteness} therefore makes the whole collection of optimal destabilizers finite.

Let \(\delta_0>0\) be the constant furnished by Theorem~\ref{thm:null-finiteness} for the shifted pair. Compactness gives \(q>0\) with \(b_0\ge q a_X\). For every nonoptimal \(p\)-fold,
\[
r(V,\alpha,\beta)-\Gamma=\frac{J(V,\alpha,\beta_0)}{(n-p)\int_V\alpha^p}
\ge\frac{\delta_0\int_V\beta_0^p}{(n-p)\int_V\alpha^p}
\ge\frac{\delta_0q^p}{n-p}.
\]
The minimum over \(1\le p<n\) is a positive \(\sigma\). This proves the asserted positive gap above the attained minimum.
\end{proof}

\begin{remark}
Theorem~\ref{thm:null-finiteness} assumes only numerical \(J\)-nefness of a K\"ahler pair on a compact K\"ahler manifold. Its uniform gap concerns reduced irreducible subvarieties. Theorem~\ref{thm:optimal-finiteness} requires a nonpositive threshold $\Gamma$ and concerns optimal destabilizers. The following examples and observations explain these restrictions.

On \(\mathbb P^n\), take \(\alpha=\beta\) equal to the hyperplane class. Then \(c=n\) and, for every proper subvariety,
\[
J(V,\alpha,\beta)=(n-p)\int_V\alpha^p,\qquad r(V)=1.
\]
Thus \(\Gamma=1\), and every proper subvariety is a minimizer. There are infinitely many hyperplanes. The positivity of \(\beta-\Gamma\alpha\) fails here (it is zero), so the change of class used in Lemma~\ref{lem:shift} is unavailable.

 For \(\Gamma<0\), subvarieties with \(\Gamma<r(V,\alpha,\beta)\le0\) are not the null subvarieties of the shifted pair. The proof controls the optimal destabilizers, and makes no finiteness assertion about that larger destabilizing collection.
\end{remark}

\AtNextBibliography{\small}
\begingroup
\setlength\bibitemsep{2pt}
\printbibliography
\endgroup
\end{document}